\documentclass[12pt,a4paper,reqno]{amsart}

\usepackage{amsmath,amssymb,amsthm,mathrsfs}
\usepackage[a4paper,left=2.5cm,right=2.5cm,top=2.4cm,bottom=2.4cm]{geometry}
\usepackage[unicode,colorlinks=false,pdfborder={0 0 0}]{hyperref}
\usepackage{microtype}

\theoremstyle{plain}
\newtheorem{theorem}{Theorem}[section]
\newtheorem{lemma}[theorem]{Lemma}
\newtheorem{proposition}[theorem]{Proposition}

\theoremstyle{remark}
\newtheorem{remark}[theorem]{Remark}
\numberwithin{equation}{section}

\makeatletter
\renewcommand{\subsection}{\@startsection{subsection}{2}{\z@}%
	{2.5ex plus .8ex minus .2ex}{1.2ex plus .2ex}{\normalfont\bfseries}}
\makeatother
\author{Changyu Ren}
\address{School of Mathematical Science\\
	Jilin University\\ Changchun\\ China}
\email{rency@jlu.edu.cn}
\author{Xianlu Sun}
\address{School of Mathematics and Statistics\\
	Northeast Normal University\\ Changchun\\ China}
\email{sunxl221@nenu.edu.cn}
\thanks{Research of the first author is supported by NSFC Grant No. 12571216, 11871243.}

\begin{document}
	
	\title[Rigidity of Convex Entire Solutions to Sum Hessian Equations]
	{Rigidity of Convex Entire Solutions to Sum Hessian Equations}
	\begingroup
	\setbox0=\hbox{\(0\)}
	\endgroup
	\begin{abstract}
		We study the rigidity of convex entire solutions to sum Hessian equations. Using a dimension-lifting identity, we reduce the problem to a standard \(k\)-Hessian equation in one higher dimension. By applying existing rigidity theorems, we prove that every \(C^2\) convex entire solution is a quadratic polynomial for \(2\leq k\leq n\), \(\alpha\geq0\), and a positive constant right-hand side, without any growth assumption at infinity. We also discuss the linear case \(k=1\).
	\end{abstract}
	
	\keywords{Sum Hessian equations, Convex entire solutions,
		Rigidity theorem, Dimension lifting}
	
	\maketitle
	
	\section{introduction}
	\label{sec:introduction}
	
	This paper studies the rigidity problem for convex entire solutions of the following sum Hessian equation in Euclidean space:
	\begin{equation}\label{eq:sum-hessian}
		S_k(D^2u)
		:=\sigma_k\bigl(\lambda(D^2u)\bigr)
		+\alpha\sigma_{k-1}\bigl(\lambda(D^2u)\bigr)
		=c_0,
		\quad x\in\mathbb R^n,
	\end{equation}
	where \(n\ge 2\), \(2\le k\le n\), \(\alpha\ge 0\), and \(c_0>0\) is a constant.
	
	To explain the notation in the equation, let \(\operatorname{Sym}(n)\) denote the space of \(n\times n\) real symmetric matrices. For any \(H\in\operatorname{Sym}(n)\), let
	\[
	\lambda(H)=\bigl(\lambda_1(H),\cdots,\lambda_n(H)\bigr)\in\mathbb R^n
	\]
	be its eigenvalue vector counted with multiplicity. In particular, for \(u\in C^2(\mathbb R^n)\), its Hessian matrix is \(D^2u=(u_{x_ix_j})_{1\le i,j\le n}\), and
	\[
	\lambda(D^2u)=
	\bigl(\lambda_1(D^2u),\cdots,\lambda_n(D^2u)\bigr)
	\]
	denotes the eigenvalue vector of \(D^2u\).
	
	For \(1\le k\le n\), we use the unnormalized \(k\)-th elementary symmetric polynomial
	\begin{equation}\label{eq:elementary}
		\sigma_k(H)=\sigma_k\bigl(\lambda(H)\bigr)
		=\sum_{1\le i_1<\cdots<i_k\le n}
		\lambda_{i_1}(H)\cdots\lambda_{i_k}(H),
		\quad 1\le k\le n,
	\end{equation}
	and set \(\sigma_0(H)=1\). When no confusion arises, we abbreviate \(\sigma_k\bigl(\lambda(H)\bigr)\) as \(\sigma_k(H)\). We write \(H\ge 0\) to mean that \(H\) is positive semidefinite; for a \(C^2\) function, convexity of \(u\) is equivalent to \(D^2u\ge 0\). By an entire solution, we mean a solution defined on all of \(\mathbb R^n\) that satisfies the corresponding equation.
	
	The central question studied in this paper is: must every convex \(C^2\) entire solution of equation~\eqref{eq:sum-hessian} be a quadratic polynomial?
	
	Hessian equations are an important class of fully nonlinear elliptic equations, and they play a significant role in convex geometry, geometric analysis, and the regularity theory of partial differential equations. Closely related to equation~\eqref{eq:sum-hessian} is the standard \(k\)-Hessian equation
	\begin{equation}\label{eq:standard-k-hessian}
		\sigma_k(D^2v)=c_0>0,
		\quad x\in\mathbb R^n.
	\end{equation}
	When \(k=1\), it is the Poisson equation with a constant right-hand side; when \(k=n\), it is the Monge--Amp\`ere equation.
	
	For \(1\le k\le N\), define the G\r{a}rding cone
	\[
	\Gamma_k^N
	=\bigl\{\lambda\in\mathbb R^N:
	\sigma_j(\lambda)>0,\quad 1\le j\le k\bigr\}.
	\]
	On this cone, the concavity of \(\sigma_k^{1/k}\) and the corresponding ellipticity are important structural properties of the standard \(k\)-Hessian equation; see~\cite{CNS1985}.
	
	Classical results first appeared at the two endpoints. For \(k=1\), the Liouville theorem for harmonic functions and convexity imply that every convex \(C^2\) entire solution of the Poisson equation with constant right-hand side must be a quadratic polynomial. For \(k=n\), the well-known J\"orgens--Calabi--Pogorelov theorem states that convex solutions of the Monge--Amp\`ere equation with positive constant right-hand side in Euclidean space have quadratic polynomial rigidity. J\"orgens~\cite{Jorgens1954} first solved the two-dimensional case; Calabi~\cite{Calabi1958} extended the result to dimensions three to five; Pogorelov~\cite{Pogorelov1972} later completed the proof in general dimension. Cheng--Yau~\cite{ChengYau1986} gave a geometric proof, and Caffarelli--Li~\cite{CaffarelliLi2003} extended the result to a weaker solution framework. For \(2\leq k<n\), the theory of elliptic cones and admissibility for Hessian equations established by Caffarelli--Nirenberg--Spruck~\cite{CNS1985} provided an important foundation.
	
	For \(2\leq k<n\), early rigidity results usually required growth conditions at infinity in addition to convexity. Bao--Chen--Guan--Ji~\cite{BaoChenGuanJi2003} proved that under quadratic lower growth, every strictly convex entire solution is a quadratic polynomial. Here quadratic lower growth means that there exist \(b>0\), \(B\in\mathbb R\), and \(R>0\) such that for all sufficiently large \(|x|\),
	\(v(x)\geq b|x|^2-B\).
	Li--Ren--Wang~\cite{LiRenWang2016} used Pogorelov-type interior second-order estimates to extend the corresponding rigidity results to \((k+1)\)-admissible entire solutions; Du~\cite{Du2022} further started from the geometric properties of sublevel sets and gave a criterion for quadratic polynomial rigidity in the locally strictly convex case. Wang--Bao~\cite{WangBao2022} studied \(k\)-admissible entire solutions with two-sided quadratic growth, and Chu--Dinew~\cite{ChuDinew2023} studied admissible solutions satisfying additional lower bounds for higher-order symmetric functions under quadratic lower growth. Later, Tu~\cite{Tu2024} established the corresponding rigidity results for admissible entire solutions of the \((n-1)\)-Hessian equation; Zhang~\cite{Zhang2025} used concavity inequalities and interior second-order estimates to prove \(k\)-Hessian rigidity under a semiconvex condition and quadratic lower growth. Recently, Wang~\cite{BinWang2026} established rigidity results for smooth \(k\)-admissible entire solutions with quadratic lower growth in the range \(2k>n\). These works gradually weakened the convexity or admissibility conditions on the solutions, but in general they still retained some growth assumption at infinity.
	
	For \(k=2\), stronger results without growth assumptions have already been obtained. Warren--Yuan~\cite{WarrenYuan2009} used Hessian estimates for the three-dimensional \(\sigma_2\) equation to obtain rigidity of entire solutions with quadratic growth. Chang--Yuan~\cite{ChangYuan2010} proved that under Hessian lower bounds close to convexity, every entire solution of \(\sigma_2(D^2v)=1\) must be a quadratic polynomial, with no growth assumption at infinity. Chen--Xiang~\cite{ChenXiang2019} later obtained rigidity results with growth conditions in three dimensions and in higher dimensions under corresponding \(\sigma_3\) lower bounds. After that, Shankar--Yuan~\cite{ShankarYuan2022} further proved that every smooth semiconvex entire solution of \(\sigma_2(D^2v)=1\) must be a quadratic polynomial, thereby removing the earlier near-convexity restriction. In terms of regularity, Shankar--Yuan~\cite{ShankarYuan2025} later established interior Hessian estimates for the four-dimensional \(\sigma_2\) equation and discussed additional structural conditions in higher dimensions; Li--Wu~\cite{LiWu2026} further established interior estimates for the quadratic Hessian equation in arbitrary dimension, and their results also imply Liouville-type rigidity under quadratic growth. For related geometric equations, Yuan~\cite{Yuan2002} studied Bernstein-type rigidity for the special Lagrangian equation; Lu--Sroka ~\cite{LuSroka2026} used a transformation to the \(\sigma_2\) equation to prove Liouville-type rigidity of semiconvex entire solutions of the Hessian quotient equation 
	\[
	\frac{\sigma_2(D^2v)}{\sigma_1(D^2v)}=1.
	\]

	On the other hand, the above results cannot be extended indiscriminately to all nonconvex \(k\)-admissible entire solutions. Warren~\cite{Warren2016} constructed non-polynomial entire \(k\)-admissible solutions, showing that admissibility alone is not enough to guarantee general Bernstein-type rigidity. In contrast, Zhou--Zhu~\cite{ZhouZhu2026} recently proved that when \(N\geq4\) and \(3\leq k\leq N-1\), every convex \(C^2\) entire solution of \(\sigma_k(D^2v)=c_0>0\) in \(\mathbb R^N\) is a quadratic polynomial, with no additional growth assumption at infinity. Liu--Ren~\cite{LiuRen2023} directly studied the sum-Hessian equation 
	\[
	\sigma_k(D^2u)+\alpha\sigma_{k-1}(D^2u)=1,
	\]
	and used Pogorelov-type estimates to establish rigidity of entire solutions under \(k\)-admissibility and quadratic growth assumptions.
	
	We obtain the following rigidity result for convex entire solutions.
	
	\begin{theorem}\label{thm:main}
		Let \(n\geq2\), \(2\leq k\leq n\), \(\alpha\geq0\), and let \(c_0>0\) be a constant. If \(u\in C^2(\mathbb R^n)\) is convex and satisfies \eqref{eq:sum-hessian}, then \(u(x)\) is a quadratic polynomial, i.e.
		\begin{equation}\label{eq:quadratic-solution}
			u(x)=\frac12 x^TQx+b\cdot x+d,
		\end{equation}
		where \(Q\in\operatorname{Sym}(n)\) is a constant positive semidefinite matrix, \(b\in\mathbb R^n\), \(d\in\mathbb R\).
	\end{theorem}
	
	This theorem requires no growth assumption at infinity and allows the Hessian matrix to be degenerate. The basic idea of its proof is to use the dimension-lifting identity
	\[
	\sigma_k\bigl(\tilde{H}\bigr)
	=\sigma_k(H)+\alpha\sigma_{k-1}(H),
	\quad H\in\operatorname{Sym}(n),
	\]
	where \(\tilde{H}=
	\begin{pmatrix}
		H & 0       \\
		0 & \alpha
	\end{pmatrix}\). Specifically, by constructing
	\[
	U(x,t)=u(x)+\frac{\alpha}{2}t^2,
	\]
	one can transform the original sum-Hessian equation exactly into the standard \(k\)-Hessian equation in one higher dimension. Since \(\alpha\ge0\), this construction preserves convexity. For \(3\le k\le n\), we apply the existing rigidity theorem of Zhou--Zhu~\cite{ZhouZhu2026}; for \(k=2\), we apply the existing rigidity theorem of Shankar--Yuan~\cite{ShankarYuan2022} together with a local regularity argument. This paper does not reprove these external rigidity theorems; instead, it focuses on verifying the dimension-lifting construction and its applicability conditions.
	
	The rest of the paper is organized as follows. Section~\ref{sec:preliminaries} introduces preliminaries. We first prove the dimension-lifting identity for elementary symmetric polynomials and the convex lifting property at the level of functions. Then we discuss the concavity and ellipticity of the lifted operator, and state the standard rigidity theorems and local regularity results needed in the proof. Section~\ref{sec:proof} first completes the proof of the main theorem for the two cases \(3\le k\le n\) and \(k=2\), and then discusses convex entire solutions for \(k=1\).
	
	\section{Preliminaries}
	\label{sec:preliminaries}
	
	This section introduces the dimension-lifting identity, function lifting, and structural properties of the operator used in this paper, and lists the existing rigidity results needed for the proof of the main theorem.
	
	\begin{lemma}
		\label{lem:lifting-identity}
		
		Let \(H\in\operatorname{Sym}(n)\), \(\alpha\in\mathbb R\), and \(1\le k\le n\). Let \(\sigma_0(H)=1\). Then
		\begin{equation}\label{eq:lifting-identity}
			\sigma_k\left(
			\begin{pmatrix}
				H & 0       \\
				0 & \alpha
			\end{pmatrix}
			\right)
			=
			\sigma_k(H)+\alpha\sigma_{k-1}(H).
		\end{equation}
		
	\end{lemma}
	
	\begin{proof}
		Let \(\lambda_1,\cdots,\lambda_n\) be the eigenvalues of \(H\), counted with multiplicity. The block diagonal matrix
		\[
		\widetilde H=
		\begin{pmatrix}
			H & 0       \\
			0 & \alpha
		\end{pmatrix}
		\]
		has eigenvalues
		\[
		\lambda_1,\cdots,\lambda_n,\alpha.
		\]
		Write \(\widetilde\lambda=(\lambda_1,\cdots,\lambda_n,\alpha)\). By the definition of elementary symmetric polynomials,
		\[
		\sigma_k(\widetilde H)
		=
		\sum_{1\le i_1<\cdots<i_k\le n+1}
		\widetilde\lambda_{i_1}\cdots\widetilde\lambda_{i_k}.
		\]
		We split this sum into two parts according to whether it contains the new eigenvalue \(\alpha\).
		
		(1) The part not containing \(\alpha\) equals
		\(
		\sigma_k(H).
		\)
		(2) The part containing \(\alpha\) is obtained by choosing \(k-1\) eigenvalues from \(\lambda_1,\cdots,\lambda_n\), and hence equals
		\(
		\alpha\sigma_{k-1}(H).
		\)
		For \(k=1\), this remains valid because \(\sigma_0(H)=1\).
		
		Adding the two parts gives
		\[
		\sigma_k(\widetilde H)
		=
		\sigma_k(H)+\alpha\sigma_{k-1}(H).
		\]
		That is,
		\[
		\sigma_k\left(
		\begin{pmatrix}
			H & 0       \\
			0 & \alpha
		\end{pmatrix}
		\right)
		=
		\sigma_k(H)+\alpha\sigma_{k-1}(H).
		\]
		This proves the lemma.
	\end{proof}
	
	\begin{lemma}
		\label{lem:lifting-map}
		
		Let \(\alpha\ge 0\) be a constant, and let \(u\in C^2(\mathbb R^n)\) be convex. Define
		\begin{equation}\label{eq:lifted-function}
			U(x,t)=u(x)+\frac{\alpha}{2}t^2,\quad (x,t)\in\mathbb R^{n+1}.
		\end{equation}
		Then \(U\in C^2(\mathbb R^{n+1})\) is convex, and
		\begin{equation}\label{eq:lifted-hessian}
			D_{x,t}^2U=
			\begin{pmatrix}
				D_x^2u & 0       \\
				0      & \alpha
			\end{pmatrix}.
		\end{equation}
		Moreover, \(u\) satisfies
		\[
		\sigma_k(D^2u)+\alpha\sigma_{k-1}(D^2u)=c_0 \quad\text{in }\mathbb R^n
		\]
		if and only if \(U\) satisfies
		\begin{equation}\label{eq:lifted-equation}
			\sigma_k(D^2U)=c_0 \quad\text{in }\mathbb R^{n+1}.
		\end{equation}
		
	\end{lemma}
	
	\begin{proof}
		Let \(u\in C^2(\mathbb R^n)\) be a convex function, and let \(\alpha\ge 0\). Define
		\[
		U(x,t)=u(x)+\frac{\alpha}{2}t^2,\quad (x,t)\in\mathbb R^n\times\mathbb R.
		\]
		Direct differentiation gives
		\[
		D_{x,t}^2U(x,t)
		=
		\begin{pmatrix}
			D_x^2u(x) & 0       \\
			0         & \alpha
		\end{pmatrix},
		\]
		that is, \eqref{eq:lifted-hessian}. Since \(u\) is convex, \(D_x^2u(x)\ge 0\). Also \(\alpha\ge 0\). Hence for any
		\(\zeta=(\xi,s)\in\mathbb R^{n+1}\),
		\[
		\zeta^TD_{x,t}^2U(x,t)\zeta
		=\xi^TD_x^2u(x)\xi+\alpha s^2\ge 0.
		\]
		Therefore \(D_{x,t}^2U\ge 0\), and so \(U\) is convex on \(\mathbb R^{n+1}\).
		
		Let \(\lambda_1(x),\cdots,\lambda_n(x)\) be the eigenvalues of \(D_x^2u(x)\). Then the eigenvalues of
		\(D_{x,t}^2U(x,t)\) are
		\(\lambda_1(x),\cdots,\lambda_n(x),\alpha\). Thus, by the definition of elementary symmetric polynomials,
		\[
		\sigma_k(D_{x,t}^2U(x,t))
		=
		\sigma_k(D_x^2u(x))
		+\alpha\sigma_{k-1}(D_x^2u(x)),
		\]
		with the convention \(\sigma_0=1\). This is exactly the lifting identity in Lemma~\ref{lem:lifting-identity}, and the right side does not depend on \(t\).
		
		If
		\[
		\sigma_k(D_x^2u)+\alpha\sigma_{k-1}(D_x^2u)=c_0
		\quad\text{in }\mathbb R^n,
		\]
		then for any \((x,t)\in\mathbb R^{n+1}\), the lifting identity gives
		\[
		\sigma_k(D_{x,t}^2U(x,t))=c_0,
		\]
		so the lifted \(k\)-Hessian equation holds.
		
		Conversely, if
		\[
		\sigma_k(D_{x,t}^2U)=c_0
		\quad\text{in }\mathbb R^{n+1},
		\]
		then take any \(x\in\mathbb R^n\). At \((x,0)\) we have
		\[
		c_0=\sigma_k(D_{x,t}^2U(x,0))
		=\sigma_k(D_x^2u(x))
		+\alpha\sigma_{k-1}(D_x^2u(x)).
		\]
		Since \(x\) is arbitrary, the original summed Hessian equation holds. Therefore
		\[
		\sigma_k(D_x^2u)+\alpha\sigma_{k-1}(D_x^2u)=c_0
		\ \text{in }\mathbb R^n
		\iff
		\sigma_k(D_{x,t}^2U)=c_0
		\ \text{in }\mathbb R^{n+1}.
		\]
		Together with regularity, the Hessian representation, and the convexity of \(U\), this proves the lemma.
	\end{proof}
	
	Suppose \(u\) is a convex entire solution of \eqref{eq:sum-hessian}. By Lemma~\ref{lem:lifting-map}, the lifted function \(U\) satisfies \(D^2U \ge 0\) and \(\sigma_k(D^2U) = c_0 > 0\). Hence, \(D^2U\) has at least \(k\) positive eigenvalues, so
	\[
	\sigma_j(D^2U) > 0, \quad 1 \le j \le k.
	\]
	It follows that the eigenvalues of the lifted Hessian belong to the G\r{a}rding cone \(\Gamma_k^{n+1}\). Moreover, the dimension-lifting identity can also be used to study the concavity of the sum Hessian operator. Define
	\begin{equation}\label{eq:lifted-cone}
		Q_\alpha = \left\{ H \in \operatorname{Sym}(n) : \lambda\bigl(\widetilde H\bigr) \in \Gamma_k^{n+1} \right\}.
	\end{equation}
	Combining the convexity of the standard \(k\)-Hessian elliptic cone and the concavity of \(\sigma_k^{1/k}\) on this cone (see~\cite{CNS1985}), we obtain the following conclusion.
	
	\begin{proposition}
		\label{prop:operator-structure}
		Let \(Q_\alpha\) be defined by \eqref{eq:lifted-cone}. Then \(Q_\alpha\) is open and convex, and
		\begin{equation}\label{eq:lifted-concavity}
			H\longmapsto S_k(H)^{1/k}=\sigma_k(\widetilde H)^{1/k}
		\end{equation}
		is concave on \(Q_\alpha\). Moreover, if \(\alpha\ge 0\), \(H\ge 0\), and \(S_k(H)>0\), then \(S_k\) is elliptic at \(H\).
		
	\end{proposition}
	
	\begin{proof}
		
		Define the matrix cone
		\[
		G_k^{n+1}
		=\bigl\{M\in\operatorname{Sym}(n+1):
		\lambda(M)\in\Gamma_k^{n+1}\bigr\},
		\]
		and consider the lifting map
		\[
		T_\alpha:\operatorname{Sym}(n)
		\longrightarrow\operatorname{Sym}(n+1),
		\quad
		T_\alpha(H)=\widetilde H.
		\]
		
		By standard \(k\)-Hessian theory, the cone
		\(G_k^{n+1}\) is open and convex, and the function
		\[
		M\longmapsto \sigma_k(M)^{1/k}
		\]
		is concave on \(G_k^{n+1}\).
		
		Since \(T_\alpha\) is continuous and
		\[
		Q_\alpha=T_\alpha^{-1}(G_k^{n+1}),
		\]
		it follows that \(Q_\alpha\) is open.
		
		We next show that \(Q_\alpha\) is convex.
		Let \(H_1,H_2\in Q_\alpha\) and \(t\in[0,1]\).
		By the definition of \(Q_\alpha\),
		\[
		T_\alpha(H_1),\,
		T_\alpha(H_2)\in G_k^{n+1}.
		\]
		Since \(T_\alpha\) is affine, we have
		\[
		\begin{aligned}
			& T_\alpha\bigl(tH_1+(1-t)H_2\bigr) \\
			& \quad
			=tT_\alpha(H_1)+(1-t)T_\alpha(H_2).
		\end{aligned}
		\]
		The convexity of \(G_k^{n+1}\) implies that
		\[
		tT_\alpha(H_1)+(1-t)T_\alpha(H_2)
		\in G_k^{n+1}.
		\]
		Consequently,
		\[
		tH_1+(1-t)H_2\in Q_\alpha,
		\]
		which proves that \(Q_\alpha\) is convex.
		
		It remains to establish the concavity of
		\(S_k^{1/k}\) on \(Q_\alpha\).
		By Lemma~\ref{lem:lifting-identity}, we have
		\[
		S_k(H)=\sigma_k\bigl(T_\alpha(H)\bigr),
		\quad H\in Q_\alpha.
		\]
		In particular,
		\[
		S_k(H)^{1/k}
		=\sigma_k\bigl(T_\alpha(H)\bigr)^{1/k}.
		\]
		
		For arbitrary \(H_1,H_2\in Q_\alpha\)
		and \(t\in[0,1]\), the affine property of
		\(T_\alpha\), together with the concavity of
		\(\sigma_k^{1/k}\) on \(G_k^{n+1}\), yields
		\[
		\begin{aligned}
			& S_k\bigl(tH_1+(1-t)H_2\bigr)^{1/k}   \\
			& =\sigma_k\Bigl(
			T_\alpha\bigl(tH_1+(1-t)H_2\bigr)
			\Bigr)^{1/k}                            \\
			& =\sigma_k\Bigl(
			tT_\alpha(H_1)+(1-t)T_\alpha(H_2)
			\Bigr)^{1/k}                            \\
			& \geq
			t\,\sigma_k\bigl(T_\alpha(H_1)\bigr)^{1/k}
			+(1-t)\,
			\sigma_k\bigl(T_\alpha(H_2)\bigr)^{1/k} \\
			& =tS_k(H_1)^{1/k}
			+(1-t)S_k(H_2)^{1/k}.
		\end{aligned}
		\]
		Therefore, \(S_k^{1/k}\) is concave on
		\(Q_\alpha\).
		
		By Li, Ren, and Wang~\cite{LiRenWang2019}, the operator is elliptic.
	\end{proof}
	
	\begin{remark}
		\label{rem:local-ellipticity}
		Proposition~\ref{prop:operator-structure} gives pointwise ellipticity at matrices satisfying the above conditions. For \(C^2\) solutions, the continuity of the Hessian and the linearization coefficients yields locally uniform ellipticity near each such point. However, without further global estimates, one cannot conclude uniform ellipticity on the whole space.
	\end{remark}
	
	Below we list the two existing rigidity theorems needed for the proof of the main theorem, and supplement them with a local regularity argument from \(C^2\) convex solutions to smooth solutions.
	
	\begin{theorem}\cite{ZhouZhu2026}
		\label{thm:zhou-zhu}
		Let \(N\geq 4\), \(3\leq k\leq N-1\), and \(c_0>0\). If \(V\in C^2(\mathbb R^N)\) satisfies
		\[
		D^2V\geq0,\quad \sigma_k(D^2V)=c_0 \quad\text{in }\mathbb R^N,
		\]
		then
		\[
		V(x)=\frac12x^TQx+b\cdot x+d,
		\]
		where \(Q\) is a constant positive semidefinite matrix, \(b\in\mathbb R^N\), and \(d\in\mathbb R\).
		
	\end{theorem}
	
	\begin{theorem}\cite{ShankarYuan2022}
		\label{thm:shankar-yuan}
		Let \(V:\mathbb R^N\to\mathbb R\) be a smooth semiconvex function satisfying
		\[
		\sigma_2(D^2V)=1 \quad\text{in }\mathbb R^N,
		\]
		and assume that there exists a constant \(K>0\) such that
		\[
		D^2V\geq-KI_N,
		\]
		where \(I_N\) denotes the \(N\times N\) identity matrix. Then \(V\) is a quadratic polynomial.
	\end{theorem}
	
	To apply Theorem~\ref{thm:shankar-yuan}, note that the positive constant \(c\) on the right-hand side can be normalized to \(1\) by scaling \(\widetilde V=c^{-1/2}V\). Since we only assume that the solution is \(C^2\), we first prove the required local regularity.
	
	\begin{lemma}
		\label{lem:regularity} Let \(V\in C^2(\mathbb R^N)\) be convex and satisfy
		\[
		\sigma_2(D^2V)=c>0 \quad\text{in }\mathbb R^N.
		\]
		Then \(V\in C^\infty(\mathbb R^N)\).
		
	\end{lemma}
	
	\begin{proof} Fix \(X_0\in\mathbb R^N\), and set \(M_0=D^2V(X_0)\). The convexity of \(V\) implies \(M_0\ge0\). Since \(\sigma_2(M_0)=c>0\), \(M_0\) has at least two positive eigenvalues.
		
		Let \(G(M)=\sigma_2(M)^{1/2}\). Using the identity
		\[
		\sigma_2(M)=\frac12\bigl[(\operatorname{tr}M)^2-\operatorname{tr}(M^2)\bigr],
		\]
		we obtain
		\[
		DG(M)=
		\frac{(\operatorname{tr}M)I-M}{2\sigma_2(M)^{1/2}}.
		\]
		Let the eigenvalues of \(M_0\) be \(\lambda_1,\cdots,\lambda_N\). Since \(\lambda_i\ge0\) and at least two eigenvalues are strictly positive, the eigenvalues of \(DG(M_0)\) satisfy
		\[
		\frac{\displaystyle\sum_{j\ne i}\lambda_j}{2\sqrt c}>0,\quad 1\le i\le N.
		\]
		Hence \(DG(M_0)\) is positive definite.
		
		By standard \(2\)-Hessian theory, \(G=\sigma_2^{1/2}\) is smooth and concave on the set of matrices corresponding to the elliptic cone \(\Gamma_2^N\). By the continuity of \(DG\), one can choose a bounded convex matrix ball \(B\) containing \(M_0\), whose closure is contained in this matrix elliptic cone, and there exist constants \(0<\lambda\le\Lambda\) such that
		\[
		\lambda I\le DG(P)\le\Lambda I,\quad P\in B.
		\]
		By the continuity of \(D^2V\), one can choose a sufficiently small neighborhood of \(X_0\) such that all Hessians in it lie in \(B\).
		
		To apply the Evans--Krylov regularity theorem, define
		\[
		\widehat G(M)=
		\inf_{P\in B}\left\{G(P)+\operatorname{tr}\!\left(DG(P)(M-P)\right)\right\},
		\quad M\in\operatorname{Sym}(N),
		\]
		Using the trace of matrix products, concavity gives
		\[
		G(M)\le G(P)+\operatorname{tr}\!\left(DG(P)(M-P)\right),\quad M,P\in B.
		\]
		Taking the infimum over \(P\in B\) and then taking \(P=M\), we obtain \(\widehat G(M)=G(M)\) for \(M\in B\). Moreover, since \(G\) and \(DG\) are bounded on \(\overline B\), \(\widehat G\) is finite-valued and, as the infimum of affine functions, is concave. By the derivative bounds above, for any \(M\in\operatorname{Sym}(N)\) and positive semidefinite matrix \(Q\),
		\[
		\lambda\operatorname{tr}Q
		\le
		\widehat G(M+Q)-\widehat G(M)
		\le
		\Lambda\operatorname{tr}Q.
		\]
		Hence \(\widehat G\) is a concave uniformly elliptic operator defined on the entire space of symmetric matrices, and in the chosen spatial neighborhood,
		\[
		\widehat G(D^2V)=G(D^2V)=\sqrt c.
		\]
		By the local Evans--Krylov regularity theorem, there exists \(\beta\in(0,1)\) such that
		\[
		V\in C_{\mathrm{loc}}^{2,\beta}
		\]
		in a neighborhood of \(X_0\).
		
		Furthermore, since \(G\) is smooth on the chosen Hessian neighborhood, we apply difference quotients to the original equation \(G(D^2V)=\sqrt c\). The coefficients of the resulting linear equation are obtained by averaging \(DG\) along line segments between nearby Hessians. Since \(B\) is convex and \(D^2V\in C^\beta\), these coefficients are uniformly elliptic and uniformly \(C^\beta\) on smaller balls. Interior Schauder estimates, followed by passing to the limit in the difference quotients, give \(V\in C_{\mathrm{loc}}^{3,\beta}\). Differentiating the original equation successively and applying Schauder estimates then gives
		\[
		V\in C_{\mathrm{loc}}^{m,\beta},\quad m\ge2.
		\]
		Since \(X_0\in\mathbb R^N\) is arbitrary, \(V\in C^\infty(\mathbb R^N)\). The proof is complete.
	\end{proof}
	
	\section{Proof of Theorem~\ref{thm:main}}
	\label{sec:proof}
	
	We now complete the proof of Theorem~\ref{thm:main}.
	
	\begin{proof}[Proof of Theorem~\ref{thm:main}] Let \(U\) be defined by \eqref{eq:lifted-function}. By Lemma~\ref{lem:lifting-map},
		\begin{equation}\label{eq:lifted-assumptions}
			U\in C^2(\mathbb R^{n+1}),\quad D^2U\geq0,\quad \sigma_k(D^2U)=c_0>0.
		\end{equation}
		First consider the case \(3\leq k\leq n\). Let \(N=n+1\). Then
		\[
		N\geq4,\quad 3\leq k\leq n=N-1.
		\]
		Hence \(U\) satisfies all the assumptions of Theorem~\ref{thm:zhou-zhu}. By Theorem \ref{thm:zhou-zhu}, \(U\) is a quadratic polynomial. In particular, when \(k=n\), after lifting we have \(k=N-1\), which is still within the range of that theorem.
		
		Next consider the case \(k=2\). By \eqref{eq:lifted-assumptions} and Lemma~\ref{lem:regularity}, \(U\in C^\infty(\mathbb R^{n+1})\). Let
		\[
		\widetilde U=c_0^{-1/2}U.
		\]
		Using the 2-homogeneity of \(\sigma_2\), we obtain
		\[
		\sigma_2(D^2\widetilde U)=c_0^{-1}\sigma_2(D^2U)=1.
		\]
		Since \(c_0>0\) and \(D^2U\geq0\), we have \(D^2\widetilde U\geq0\). Hence \(\widetilde U\) is convex, and in particular semiconvex. Therefore, by Theorem~\ref{thm:shankar-yuan}, \(\widetilde U\) is a quadratic polynomial, and so is \(U=c_0^{1/2}\widetilde U\).
		
		In summary, for any \(2\leq k\leq n\), the lifted function \(U\) is a quadratic polynomial. Restricting it to the hyperplane \(t=0\), we obtain from \eqref{eq:lifted-function}
		\[
		u(x)=U(x,0)=\frac12x^TQx+b\cdot x+d,
		\]
		where \(Q\in\operatorname{Sym}(n)\) is a constant matrix, \(b\in\mathbb R^n\), and \(d\in\mathbb R\). The theorem is proved.
	\end{proof}
	
	\begin{remark}
		\label{rem:degenerate-lifting} When \(\alpha=0\), the lifted function \(U\) does not depend on the new variable \(t\), and its Hessian has a zero eigenvalue; when \(\alpha>0\), the eigenvalue corresponding to the new direction is \(\alpha\). In both cases, the convexity requirements of the cited rigidity theorems are satisfied, because those theorems allow the Hessian to be positive semidefinite rather than strictly positive definite.
		
	\end{remark}
	
	\begin{remark}
		When \(k=1\), since \(\sigma_0=1\) and \(\sigma_1(D^2u)=\Delta u\), the original equation becomes
		\begin{equation}\label{eq:linear-case}
			\Delta u+\alpha=c_0.
		\end{equation}
		Let \(u\in C^2(\mathbb R^n)\) be a convex entire solution. Then \(D^2u\geq0\), so \(\Delta u\geq0\). Therefore, when \(c_0<\alpha\), equation \eqref{eq:linear-case} has no convex entire solution.
		
		When \(c_0=\alpha\), \eqref{eq:linear-case} gives \(\Delta u=0\). Combining \(D^2u\geq0\) and \(\operatorname{tr}(D^2u)=0\), this implies \(D^2u=0\), so \(u\) must be an affine function.
		
		When \(c_0>\alpha\), let \(C=c_0-\alpha>0\). Then \(\Delta u=C\). Since \(D^2u\geq0\) and the sum of its eigenvalues is identically \(C\), we have
		\[
		0\leq D^2u\leq CI.
		\]
		On the other hand, let \(w(x)=u(x)-\frac{C}{2n}|x|^2\). Then \(\Delta w=0\). By the smoothness of harmonic functions, \(u\in C^\infty(\mathbb R^n)\). Moreover, for any \(1\leq i,j\leq n\),
		\[
		\Delta u_{ij}=\partial_{ij}(\Delta u)=0.
		\]
		From \(0\leq D^2u\leq CI\), we obtain \(|u_{ij}(x)|\leq C\) on \(\mathbb R^n\). Hence each component of the Hessian is a bounded entire harmonic function. By Liouville's theorem for harmonic functions, each \(u_{ij}\) is constant, that is, \(D^2u=Q\), where \(Q\) is a constant positive semidefinite matrix. Integrating gives
		\[
		u(x)=\frac12x^TQx+b\cdot x+d,
		\]
		and
		\(
		Q\geq0,\quad \operatorname{tr}Q=c_0-\alpha.
		\)
		
	\end{remark}


\begin{thebibliography}{99}
		
		\bibitem{BaoChenGuanJi2003}
		J.~Bao, J.~Chen, B.~Guan, and M.~Ji,
		Liouville property and regularity of a Hessian quotient equation,
		\emph{Amer. J. Math.}, 125 (2003), 301--316.
		
		\bibitem{CaffarelliLi2003}
		L.~Caffarelli and Y.~Li,
		An extension to a theorem of J\"orgens, Calabi, and Pogorelov,
		\emph{Comm. Pure Appl. Math.}, 56 (2003), 549--583.
		
		\bibitem{CNS1985}
		L.~Caffarelli, L.~Nirenberg, and J.~Spruck,
		The Dirichlet problem for nonlinear second-order elliptic equations.
		III. Functions of the eigenvalues of the Hessian,
		\emph{Acta Math.}, 155 (1985), 261--301.
		
		\bibitem{Calabi1958}
		E.~Calabi,
		Improper affine hyperspheres of convex type and a generalization
		of a theorem by K.~J\"orgens,
		\emph{Michigan Math. J.}, 5 (1958), 105--126.
		
		\bibitem{ChangYuan2010}
		S.-Y.~A.~Chang and Y.~Yuan,
		A Liouville problem for the sigma-2 equation,
		\emph{Discrete Contin. Dyn. Syst.}, 28 (2010), 659--664.
		
		\bibitem{ChenXiang2019}
		L.~Chen and N.~Xiang,
		Rigidity theorems for the entire solutions of 2-Hessian equation,
		\emph{J. Differential Equations}, 267 (2019), 5202--5219.
		
		\bibitem{ChengYau1986}
		S.-Y.~Cheng and S.-T.~Yau,
		Complete affine hypersurfaces. Part I.
		The completeness of affine metrics,
		\emph{Comm. Pure Appl. Math.}, 39 (1986), 839--866.
		
		\bibitem{ChuDinew2023}
		J.~Chu and S.~Dinew,
		Liouville theorem for a class of Hessian equations,
		\emph{arXiv preprint}, arXiv:2306.13825 (2023).
		
		\bibitem{Du2022}
		S.-Z.~Du,
		Necessary and sufficient conditions to Bernstein theorem
		of a Hessian equation,
		\emph{Trans. Amer. Math. Soc.}, 375 (2022), 4873--4892.
		
		\bibitem{Jorgens1954}
		K.~J\"orgens,
		\"Uber die L\"osungen der Differentialgleichung \(rt-s^2=1\),
		\emph{Math. Ann.}, 127 (1954), 130--134.
		
		\bibitem{LiRenWang2019}
		C.~Li, C.~Ren, and Z.~Wang,
		Curvature estimates for convex solutions of some
		fully nonlinear Hessian-type equations,
		\emph{Calc. Var. Partial Differential Equations},
		58 (2019), no.~6, Article 188.
		
		\bibitem{LiRenWang2016}
		M.~Li, C.~Ren, and Z.~Wang,
		An interior estimate for convex solutions and a rigidity theorem,
		\emph{J. Funct. Anal.}, 270 (2016), 2691--2714.
		
		\bibitem{LiWu2026}
		Z.~Li and K.~Wu,
		Interior Hessian estimates for the quadratic Hessian equation,
		\emph{arXiv preprint}, arXiv:2608.23233v2 (2026).
		
		\bibitem{LiuRen2023}
		Y.~Liu and C.~Ren,
		Pogorelov type \(C^2\) estimates for sum Hessian equations
		and a rigidity theorem,
		\emph{J. Funct. Anal.}, 284 (2023), no.~1, 109726.
		
		\bibitem{LuSroka2026}
		S.~Lu and M.~Sroka,
		On Liouville's theorem for the Hessian quotient equation
		\(\sigma_2/\sigma_1\),
		\emph{arXiv preprint}, arXiv:2602.14946 (2026).
		
		\bibitem{Pogorelov1972}
		A.~V.~Pogorelov,
		On the improper convex affine hyperspheres,
		\emph{Geom. Dedicata}, 1 (1972), 33--46.
		
		\bibitem{ShankarYuan2022}
		R.~Shankar and Y.~Yuan,
		Rigidity for general semiconvex entire solutions
		to the sigma-2 equation,
		\emph{Duke Math. J.}, 171 (2022), no.~15, 3201--3214.
		
		\bibitem{ShankarYuan2025}
		R.~Shankar and Y.~Yuan,
		Hessian estimates for the sigma-2 equation in dimension four,
		\emph{Ann. of Math.}, 201 (2025), no.~2, 489--513.
		
		\bibitem{Tu2024}
		Q.~Tu,
		Pogorelov type estimates for \((n-1)\)-Hessian equations
		and related rigidity theorems,
		\emph{arXiv preprint}, arXiv:2405.02939 (2024).
		
		\bibitem{BinWang2026}
		B.~Wang,
		A Liouville type rigidity result for entire solutions
		of the \(k\)-Hessian equation,
		\emph{arXiv preprint}, arXiv:2608.29392v2 (2026).
		
		\bibitem{WangBao2022}
		C.~Wang and J.~Bao,
		Liouville property and existence of entire solutions
		of Hessian equations,
		\emph{Nonlinear Anal.}, 223 (2022), 113020.
		
		\bibitem{Warren2016}
		M.~Warren,
		Nonpolynomial entire solutions to \(\sigma_k\) equations,
		\emph{Comm. Partial Differential Equations}, 41 (2016), 848--853.
		
		\bibitem{WarrenYuan2009}
		M.~Warren and Y.~Yuan,
		Hessian estimates for the \(\sigma_2\) equation in dimension 3,
		\emph{Comm. Pure Appl. Math.}, 62 (2009), 305--321.
		
		\bibitem{Yuan2002}
		Y.~Yuan,
		A Bernstein problem for special Lagrangian equations,
		\emph{Invent. Math.}, 150 (2002), 117--125.
		
		\bibitem{Zhang2025}
		R.~Zhang,
		\(C^2\) estimates for \(k\)-Hessian equations and a rigidity theorem,
		\emph{Adv. Math.}, 480 (2025), Article 110488.
		
		\bibitem{ZhouZhu2026}
		X.~Zhou and R.~Zhu,
		Rigidity for convex entire solutions to the \(k\)-Hessian equation,
		\emph{arXiv preprint}, arXiv:2609.20445v1 (2026).
		
	\end{thebibliography}
\end{document}